\documentclass{louloupart1}
\usepackage{pinlabel}
\usepackage{graphicx}
\usepackage{soul}
\usepackage{tikz}
\usetikzlibrary{decorations.markings,arrows,arrows.meta,positioning}
\tikzset{->-/.style={decoration={
  markings,
  mark=at position #1 with {\arrow{Computer Modern Rightarrow[length=5pt,width=5pt]}}},postaction={decorate}}}
\tikzset{->-rev/.style={decoration={
  markings,
  mark=at position #1 with {\arrow{Computer Modern Rightarrow[length=5pt,width=5pt,reversed]}}},postaction={decorate}}}

\title[Endomorphisms of Artin groups of type $\tilde A_n$]{Endomorphisms of Artin groups of type $\tilde A_n$}
\author[L Paris]{Luis Paris}
\givenname{Luis}
\surname{Paris}
\address{Luis Paris, IMB, UMR 5584, CNRS, Université  
Bourgogne Europe, 21000 Dijon, France
}
\email{lparis@u-bourgogne.fr}
\author[I Soroko]{Ignat Soroko}
\givenname{Ignat}
\surname{Soroko}
\address{Ignat Soroko, Department of Mathematics, University of North Texas, Denton, TX 76203-5017, USA}
\email{ ignat.soroko@unt.edu, ignat.soroko@gmail.com}
\newtheorem{thm}{Theorem}[section]
\newtheorem{lem}[thm]{Lemma}
\newtheorem{prop}[thm]{Proposition}
\newtheorem{corl}[thm]{Corollary}
\newtheorem{conj}[thm]{Conjecture}

\theoremstyle{definition}

\newtheorem{rem}{Remark}

\newtheorem*{acknow}{Acknowledgments}

\numberwithin{equation}{section}

\makeatletter
\renewcommand{\thefigure}{\ifnum \c@section>\z@ \thesection.\fi
 \@arabic\c@figure}
\@addtoreset{figure}{section}
\makeatother

\begin{document}

\raggedbottom

\def\N{\mathbb N} \def\conj{{\rm conj}} \def\Aut{{\rm Aut}}
\def\Inn{{\rm Inn}} \def\Out{{\rm Out}} \def\Z{\mathbb Z}
\def\id{{\rm id}} \def\Im{{\rm Im}} \def\supp{{\rm supp}}
\def\Ker{{\rm Ker}} \def\PP{\mathcal P} \def\Homeo{{\rm Homeo}}
\def\SHomeo{{\rm SHomeo}} \def\LHomeo{{\rm LHomeo}}
\def\MM{\mathcal M} \def\CC{\mathcal C} \def\AA{\mathcal A}
\def\S{\mathbb S} \def\FF{\mathcal F} \def\SS{\mathcal S}
\def\LL{\mathcal L} \def\D{\mathbb D} \def\Fix{{\rm Fix}}

\newcommand{\card}{\operatorname{Card}}
\newcommand{\Sym}{\operatorname{Sym}}
\newcommand{\rel}{\operatorname{rel}}
\newcommand{\genus}{\operatorname{genus}}
\renewcommand{\le}{\leqslant}
\renewcommand{\ge}{\geqslant}

%%%%%%%%%%

\begin{abstract}
We determine a classification of the endomorphisms of the Artin group of affine type $\tilde A_n$ for $n\ge 4$. 

\smallskip\noindent
{\bf AMS Subject Classification\ \ } 
Primary: 20F36, secondary: 57K20.

\smallskip\noindent
{\bf Keywords\ \ } 
Artin groups of type $\tilde A_n$, endomorphisms, automorphisms.

\end{abstract}

\maketitle

%%%%%%%%%

\section{Introduction}\label{sec1}

Let $S$ be a finite set.
A \emph{Coxeter matrix} over $S$ is a square matrix $M=(m_{s,t})_{s,t\in S}$ indexed by the elements of $S$, with coefficients in $\N \cup \{\infty\}$, such that $m_{s,s}=1$ for all $s \in S$, and $m_{s,t} = m_{t,s} \ge 2$ for all $s,t\in S$, $s\neq t$.
Such a matrix is usually represented by a labeled graph, $\Gamma$, called a \emph{Coxeter graph}, defined by the following data.
The set of vertices of $\Gamma$ is $S$.
Two vertices $s,t\in S$ are connected by an edge if $m_{s,t}\ge 3$, and this edge is labeled with $m_{s,t}$ if $m_{s,t} \ge 4$.

If $a,b$ are two letters and $m$ is an integer $\ge 2$, then we denote by $\Pi(a,b,m)$ the alternating word $aba\ldots$ of length $m$.
In other words, $\Pi(a,b,m) = (ab)^{\frac{m}{2}}$ if $m$ is even, and $\Pi(a,b,m) = (ab)^{\frac{m-1}{2}} a$ if $m$ is odd.
Let $\Gamma$ be a Coxeter graph and let $M=(m_{s,t})_{s,t\in S}$ be its Coxeter matrix.
With $\Gamma$ we associate a group, $A[\Gamma]$, called the \emph{Artin group} of $\Gamma$, defined by the presentation
\[
A[\Gamma] = \langle S \mid \Pi(s,t,m_{s,t}) = \Pi(t,s,m_{s,t}) \text{ for } s, t \in S\,,\ s \neq t\,,\ m_{s,t}\neq \infty \rangle\,.
\]
The \emph{Coxeter group} of $\Gamma$, denoted by $W[\Gamma]$, is the quotient of $A[\Gamma]$ by the relations $s^2 = 1$, $s \in S$.

Despite the popularity of Artin groups, little is known about their automorphisms and even less about their endomorphisms.
The most studied cases are the braid groups, corresponding to the Coxeter graphs $A_n$ ($n\ge 1$), and the right-angled Artin groups.
The automorphism group of $A[A_n]$ was determined in~\cite{DyeGro1} and the set of its endomorphisms in~\cite{Caste1} for $n\ge 5$, in~\cite{ChKoMa1} for $n \ge 4$ and in~\cite{Orevk1} for $n\ge2$ (see also~\cite{BelMar1,KorMar1}).
On the other hand, there are many papers dealing with automorphism groups of right-angled Artin groups (see~\cite{BrChVo1,ChaVog1,ChaVog2,Day1,Day2,Laure1} for example), but little is known about their endomorphisms.

Apart from these two classes, the Artin groups for which the automorphism group is determined are the $2$-generator Artin groups~\cite{GiHoMeRa1}, some $2$-dimensional Artin groups~\cite{Crisp1,AnCho1}, the large-type free-of-infinity Artin groups~\cite{Vasko1}, the Artin groups of type $B_n$, $\tilde A_n$ and $\tilde C_n$~\cite{ChaCri1}, the Artin group of type $D_4$~\cite{Sorok1}, and the Artin groups of type $D_n$ for $n\ge 6$~\cite{CasPar1}.
On the other hand, apart from Artin groups of type $A_n$, 
the set of endomorphisms is determined only for Artin groups of type $D_n$ for $n\ge6$~\cite{CasPar1}.

The aim of the present paper is to determine a classification of the endomorphisms of the Artin groups of type $\tilde A_n$ for $n\ge 4$, where $\tilde A_n$ is the affine Coxeter graph depicted in Figure~\ref{fig1_1} (see Theorem~\ref{thm2_1}). 

\begin{figure} %[ht!]
\begin{center}
\includegraphics[width=4cm]{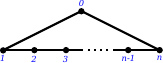}
\caption{Coxeter graph $\tilde A_n$, $n\ge2$}\label{fig1_1}
\end{center}
\end{figure}

A group $G$ is called \emph{Hopfian} if every surjective homomorphism $\psi\colon G \to G$ is injective, and it is called \emph{co-Hopfian} if every injective homomorphism $\psi\colon G \to G$ is surjective.
An easy consequence of our main theorem is that $A[\tilde A_n]$ is both Hopfian and co-Hopfian (see Corollary~\ref{corl2_5}).
The fact that all finite index subgroups (including $A[\tilde A_n]$) of the mapping class group of the punctured sphere are co-Hopfian was proved in~\cite[Corollary~3]{BelMar2}. It is also 
known that the mapping class group of a compact connected orientable surface of genus $g\ge 2$ is co-Hopfian~\cite{IvaMcc1}, and that a virtual braid group is co-Hopfian~\cite{BelPar1}. On the other hand, spherical type Artin groups are never co-Hopfian, since they admit transvection endomorphisms which are injective but not surjective, see~\cite{BelMar1}, but some of them (of types $A_n$, $B_n$, $D_n$) are known to be almost co-Hopfian, i.e.\ their quotients by the center are co-Hopfian~\cite{BelMar1,BelMar2,CasPar1}.

The results of the present paper are also used in the subsequent paper~\cite{ParSor1} by the authors, where a classification of endomorphisms of the Artin groups of type $B_n$, for $n\ge5$, and of their quotients by the center, is determined.

The paper is organized as follows.
In Section~\ref{sec2} we give definitions and precise statements of our results.
Section~\ref{sec3} contains preliminaries and Section~\ref{sec4} contains the proofs. 

\begin{acknow}
This work was started at the AIM workshop ``Geometry and topology of Artin groups'' organized by Ruth Charney, Kasia Jankiewicz, and Kevin Schreve in September 2023, and continued at the SLMath workshop ``Hot Topics: Artin Groups and Arrangements -- Topology, Geometry, and Combinatorics'' organized by Christin Bibby, Ruth Charney, 
Giovanni Paolini, and Mario Salvetti in March 2024. We thank the organizers, the AIM and SLMath. We also thank the referee who carefully read the manuscript and suggested several improvements that increased clarity and readability of the exposition. 
The first author is partially supported by the French project ``AlMaRe'' (ANR-19-CE40-0001-01) of the ANR.
The second author acknowledges support from the AMS--Simons travel grant.
\end{acknow}

%%%%%%%%%%

\section{Definitions and statements}\label{sec2}

Two other Coxeter graphs play an important role in our study: the Coxeter graphs $A_n$ and $B_n$ depicted in Figure~\ref{fig2_1}.

\begin{figure}[ht!]
\begin{center}
\begin{tabular}{cc}
\quad\includegraphics[width=4cm]{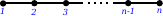}\quad&
\quad\includegraphics[width=4cm]{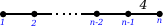}\quad\\
$A_n$&$B_n$
\end{tabular}
\caption{Coxeter graphs $A_n$ and $B_n$}\label{fig2_1}
\end{center}
\end{figure}

The standard generators of $A[A_n]$ will be denoted by $s_1,\dots,s_n$, the standard generators of $A[B_n]$ will be denoted by $r_1,\dots,r_n$, and the standard generators of $A[\tilde A_n]$ will be denoted by $t_0,t_1,\dots,t_n$.

Let $\Gamma$ be a Coxeter graph.
For $X\subset S$ we denote by $\Gamma_X$ the full subgraph of $\Gamma$ spanned by $X$, by $A_X[\Gamma]$ the subgroup of $A[\Gamma]$ generated by $X$, and by $W_X[\Gamma]$ the subgroup of $W[\Gamma]$ generated by $X$.
We know by~\cite{Lek1} that $A_X[\Gamma]$ is naturally isomorphic to $A[\Gamma_X]$, and we know by~\cite[Chapter~4,\,Section~1.8,\,Theorem~2(i)]{Bourb1} that $W_X[\Gamma]$ is naturally isomorphic to $ W[\Gamma_X]$.
A subgroup of the form $A_X[\Gamma]$ is called a \emph{standard parabolic subgroup} of $A[\Gamma]$, and a subgroup of the form $W_X[\Gamma]$ is called a \emph{standard parabolic subgroup} of $W[\Gamma]$.

The word length of an element $w$ in $W[\Gamma]$ with respect to $S$ is denoted by $\lg(w)$.
A \emph{reduced expression} of $w$ is a word $s_1s_2\ldots s_\ell$ over $S$ representing $w$ such that $\ell=\lg(w)$.
We denote by $\omega\colon A[\Gamma]\to W[\Gamma]$ the natural epimorphism which sends $s$ to $s$ for all $s\in S$.
This epimorphism admits a natural set-section $\tau\colon W[\Gamma]\to A[\Gamma]$ defined as follows.
Let $w\in W[\Gamma]$, and let $s_1s_2\ldots s_\ell$ be a reduced expression of $w$.
Then $\tau(w)$ is the element of $A[\Gamma]$ represented by $s_1s_2\ldots s_\ell$.
By~\cite{Tits1} the definition of $\tau(w)$ does not depend on the choice of the reduced expression.

We say that $\Gamma$ is of \emph{spherical type} if $W[\Gamma]$ is finite.
In this case $W[\Gamma]$ has a unique element of maximal length, denoted by $w_0$, and this element satisfies $w_0^2=1$ and $w_0Sw_0=S$ (see~\cite[Lemma~4.6.1]{Davis1}).
If $\Gamma$ is of spherical type, then the \emph{Garside element} of $A[\Gamma]$, denoted by $\Delta = \Delta [\Gamma]$, is defined to be $\tau(w_0)$.
More generally, if $\Gamma$ is any Coxeter graph and $X\subset S$ is such that $\Gamma_X$ is of spherical type, then we denote by $\Delta_X=\Delta_X[\Gamma]$ the Garside element of $A_X[\Gamma]=A[\Gamma_X] \subset A[\Gamma]$.
If $\Gamma$ is connected and of spherical type, then the center of $A[\Gamma]$, denoted by $Z(A[\Gamma])$, is an infinite cyclic group generated by either $\Delta$ or $\Delta^2$ (see~\cite{BriSai1}).
If $\Gamma$ is connected and not of spherical type, then it is conjectured that $A[\Gamma]$ has trivial center.
This conjecture is proved in many cases but remains open in the whole generality.

The Coxeter graphs $B_n$ and $A_n$ are of spherical type, while the Coxeter graph $\tilde A_n$ is not of spherical type (it is of \emph{affine} type).
We also know that the center of $A[\tilde A_n]$ is trivial (see~\cite[Proposition 1.3]{ChaPei1}).
The Garside element of $A[A_n]$ is
\[
\Delta=\Delta[A_n]=(s_1s_2\ldots s_n)(s_1s_2\ldots s_{n-1})\ldots (s_1s_2)s_1\,,
\]
see~\cite[Theorem~1.24]{KasTur1}; we have $\Delta s_i \Delta^{-1} = s_{n+1-i}$ for all $1\le i\le n$, and so $Z(A[A_n])$ is generated by $\Delta^2$.
The Garside element of $A[B_n]$ is
\[
\Delta[B_n]=(r_1\ldots r_n)^n\,,
\]
and it generates the center of $A[B_n]$, see~\cite[Satz~7.2]{BriSai1} and \cite[Chapter\,VI,\,Section~4,\,n$^\circ5$,\,(III)]{Bourb1}.

If $G$ is a group and $g\in G$, then we denote by $\conj_g\colon G\to G$, $h \mapsto g h g^{-1}$, the conjugation by $g$.
We have a homomorphism $\conj\colon G\to\Aut(G)$, $g\mapsto\conj_g$, whose image is the group $\Inn(G)$ of \emph{inner automorphisms}, and whose kernel is the center of $G$.
The group $\Inn(G)$ is a normal subgroup of $\Aut(G)$, and the quotient $\Out(G)=\Aut(G)/\Inn(G)$ is the \emph{outer automorphism group} of $G$.
Two homomorphisms $\varphi_1,\varphi_2\colon G\to H$ are said to be \emph{conjugate} if there exists $h\in H$ such that $\varphi_2=\conj_h\circ\varphi_1$.
Our classification of endomorphisms of $A[\tilde A_n]$ will be made up to conjugation.

Three automorphisms $\zeta,\eta,\mu\in\Aut(A[\tilde A_n])$ play a central role in our study.
These are defined by
\[
\zeta(t_i)=t_{i+1}\text{ for }0\le i\le n\,,\
\eta(t_i)=t_{n-i}\text{ for }0\le i\le n\,,\
\mu(t_i)=t_i^{-1}\text{ for }0\le i\le n\,,
\]
where the indices are taken modulo $n+1$.
Note that $\zeta,\eta$ generate the automorphism group of the Coxeter graph $\tilde A_n$, denoted by $\Aut(\tilde A_n)$, which is isomorphic to the dihedral group $D_{2(n+1)}$ of order $2(n+1)$.
The subgroup of $\Aut(A[\tilde A_n])$ generated by $\zeta,\eta,\mu$ is isomorphic to $\Aut(\tilde A_n)\times\Z/2\Z$.
It will be denoted by $\Aut^*(\tilde A_n)$.

\begin{rem}
We know that $A[\tilde A_n]$ embeds into $A[B_{n+1}]$ (see~\cite{KenPei1} and/or Proposition~\ref{prop3_1} below), and that $A[B_{n+1}]$ is torsion-free (see~\cite{BriSai1,Delig1}), hence $A[\tilde A_n]$ is torsion-free.
We also know that the center of $A[\tilde A_n]$ is trivial, hence $\Inn(A[\tilde A_n])\simeq A[\tilde A_n]$ is torsion-free.
Since $\Aut^*(\tilde A_n)$ is finite, it follows that $\Aut^*(\tilde A_n)\cap\Inn(A[\tilde A_n])=\{1\}$, hence $\Aut^*(\tilde A_n)$ embeds into $\Out(A[\tilde A_n])$ via the projection $\Aut(A[\tilde A_n])\to\Out(A[\tilde A_n])$.
\end{rem}

A homomorphism $\varphi\colon G\to H$ is called \emph{abelian} (resp. \emph{cyclic}) if its image is an abelian (resp. cyclic) subgroup of $H$.
A homomorphism $\varphi\colon A[A_n]\to H$ is abelian, if and only if it is cyclic, if and only if there exists $h\in H$ such that $\varphi(s_i)=h$ for all $1\le i\le n$.
Similarly, a homomorphism $\varphi\colon A[\tilde A_n]\to H$ is abelian, if and only if it is cyclic, if and only if there exists $h\in H$ such that $\varphi(t_i)=h$ for all $0\le i\le n$.

We set $Y=\{t_1,t_2,\dots,t_n\}$.
Notice that the full subgraph of $\tilde A_n$ spanned by $Y$ is isomorphic to $A_n$.
We denote by $\Delta_Y=\Delta_Y[\tilde A_n]$ the Garside element of $A_Y[\tilde A_n]$.
Furthermore, we set $\rho=t_1t_2\ldots t_n\in A_Y[\tilde A_n]$ and $\rho'=t_1^{-1}t_2^{-1}\ldots t_n^{-1}\in A_Y[\tilde A_n]$.
A direct calculation shows that $\rho t_i\rho^{-1}=\rho't_i\rho'^{-1}=t_{i+1}$ for all $1\le i\le n-1$ and $\rho^2 t_n\rho^{-2}=\rho'^2t_n\rho'^{-2}=t_1$.
Let $v_0=\rho t_n\rho^{-1}$ and $v_1=\rho't_n\rho'^{-1}$.
Then, for each $p\in\Z$, we have endomorphisms $\alpha_p,\beta_p\colon A[\tilde A_n]\to A[\tilde A_n]$ defined by
\[
\alpha_p(t_i)=\beta_p(t_i)=t_i\Delta_Y^{2p}\text{ for }1\le i\le n\,,\
\alpha_p(t_0)=v_0\Delta_Y^{2p}\,,\
\beta_p(t_0)=v_1\Delta_Y^{2p}\,.
\]
Note that $\Im(\alpha_p),\Im(\beta_p)\subset A_Y[\tilde A_n]\simeq A[A_n]$. Figure~\ref{fig_v01} depicts the standard generators $t_1,\dots,t_n$ and elements $v_0$ and $v_1$ interpreted as braids on $n+1$ strands. We note that, by Corollary~\ref{corl2_3} below, endomorphisms $\alpha_p$ and $\beta_p$ are not injective.

\begin{figure}[htb]
\begin{center}
\begin{tikzpicture}[very thick]

% x||
\begin{scope}[xscale=0.75,yscale=0.75]
\draw (0,0)--(1,3);
\draw (0,3)--(0.4,1.8); \draw (0.6,1.2)--(1,0);
\draw (1.5,0)--(1.5,3);
\draw (2.25,0)--(2.25,3);
\draw (1.25,-0.5) node {$t_1$};
\end{scope}

% |x|
\begin{scope}[xshift=3.125cm]
\begin{scope}[xscale=0.75,yscale=0.75]
\draw (1,0)--(2,3);
\draw (1,3)--(1.4,1.8); \draw (1.6,1.2)--(2,0);
\draw (0.5,0)--(0.5,3);
\draw (2.5,0)--(2.5,3);
\draw (1.5,-0.5) node {$t_2$};
\end{scope}
\end{scope}

% ||x
\begin{scope}[xshift=6.25cm]
\begin{scope}[xscale=0.75,yscale=0.75]
\draw (2,0)--(3,3);
\draw (2,3)--(2.4,1.8); \draw (2.6,1.2)--(3,0);
\draw (0.75,0)--(0.75,3);
\draw (1.5,0)--(1.5,3);
\draw (1.75,-0.5) node {$t_3$};
\end{scope}
\end{scope}

% v0
\begin{scope}[xshift=10.25cm]
\begin{scope}[xscale=0.5625,yscale=0.75]
\draw (1,0)--(1,3);
\draw (2,0)--(2,3);
\draw (0,3)--(0.8,2.2); \draw (1.2,1.8)--(1.4,1.6); \draw (1.6,1.4)--(1.8,1.2); \draw (2.2,0.8)--(3,0);
\draw (0,0)--(0.8,0.8); \draw (1.2,1.2)--(1.8,1.8); \draw (2.2,2.2)--(3,3);
\draw (1.5,-0.5) node {$v_0$};
\end{scope}
\end{scope}

% v1
\begin{scope}[xshift=13.75cm]
\begin{scope}[xscale=0.5625,yscale=0.75]
\draw (0,0)--(3,3);
\draw (0,3)--(1.4,1.6); \draw (1.6,1.4)--(3,0);
\draw (1,3)--(1,2.2); \draw (1,1.8)--(1,1.2); \draw (1,0.8)--(1,0);
\draw (2,3)--(2,2.2); \draw (2,1.8)--(2,1.2); \draw (2,0.8)--(2,0);
\draw (1.5,-0.5) node {$v_1$};
\end{scope}
\end{scope}

\end{tikzpicture}
\caption{Braid pictures of the standard generators $t_1,\dots,t_n$ and elements $v_0$ and $v_1$ used in the definition of endomorphisms $\alpha_p$ and $\beta_p$, depicted here for $n=3$.\label{fig_v01}}
\end{center}
\end{figure}

The main result of the present paper is the following.

\begin{thm}\label{thm2_1}
Let $n\ge 4$.
Let $\varphi\colon A[\tilde A_n]\to A[\tilde A_n]$ be an endomorphism.
Then we have one of the following four possibilities up to conjugation.
\begin{itemize}
\item[(1)]
$\varphi$ is cyclic.
\item[(2)]
$\varphi\in\Aut^*(\tilde A_n)$.
\item[(3)]
There exist $p\in\Z$ and $\psi\in\Aut^*(\tilde A_n)$ such that $\varphi=\psi\circ\alpha_p$.
\item[(4)]
There exist $p\in\Z$ and $\psi\in\Aut^*(\tilde A_n)$ such that $\varphi=\psi\circ\beta_p$.
\end{itemize}
\end{thm}

The reader may wonder whether there is redundancy in Theorem~\ref{thm2_1}.
The answer is essentially no.
Indeed, it will be shown in Corollary~\ref{corl2_3} that neither $\alpha_p$ nor $\beta_p$ is surjective, hence we cannot have both Case~(2) and Case~(3) together, or both Case~(2) and Case~(4) together.
Neither $\alpha_p$ nor $\beta_p$ is cyclic, hence we cannot have both Case~(1) and Case~(3) together, or both Case~(1) and Case~(4) together.
Also, an automorphism of $A[\tilde A_n]$ is never cyclic, so we cannot have both Case~(1) and Case~(2) together.
It remains to understand Case~(3) and Case~(4).

Let $\Fix (Y)$ be the subgroup of $\Aut(A[\tilde A_n])$ of automorphisms of $A[\tilde A_n]$ which restrict to the identity on $A_Y [\tilde A_n]$. If $\psi \in \Fix (Y)$, then $\psi\circ\alpha_p=\alpha_p$ and $\psi\circ\beta_p=\beta_p$. We note that $\Fix(Y)$ is not trivial since, for example, it contains $\conj_{\Delta_Y}\circ\zeta\circ\eta$. In particular, we have $\conj_{\Delta_Y}\circ\zeta\circ\eta\circ\alpha_p=\alpha_p$, which means that $\alpha_p$ is conjugate to $(\zeta\circ\eta)\circ\alpha_p$ with $\zeta\circ\eta\in\Aut^*(\tilde A_n)$ (and the same reasoning applies for $\beta_p$). This shows that the automorphism $\psi$ in Case~(3) or in Case~(4) is not unique.
However, the number $p$ in Case~(3) or in Case~(4) is unique, and we cannot have both Case~(3) and Case~(4) together, as shown by the following proposition.

\begin{prop}\label{prop2_2}
Let $n\ge 4$.
\begin{itemize}
\item[(1)]
Let $\psi,\psi'\in\Aut(A[\tilde A_n])$ and $p,q\in\Z$.
If $\psi\circ\alpha_p=\psi'\circ\alpha_q$, then $p=q$.
\item[(2)]
Let $\psi,\psi'\in\Aut(A[\tilde A_n])$ and $p,q\in\Z$.
If $\psi\circ\beta_p=\psi'\circ\beta_q$, then $p=q$.
\item[(3)]
Let $\psi,\psi'\in\Aut(A[\tilde A_n])$ and $p,q\in\Z$.
Then $\psi\circ\alpha_p\neq\psi'\circ\beta_q$.
\end{itemize}
\end{prop}

The proof of Proposition~\ref{prop2_2} will follow that of Theorem~\ref{thm2_1} in Section~\ref{sec4}.

We turn now to show some notable consequences of Theorem~\ref{thm2_1} before moving on to the preliminaries in Section~\ref{sec3} and to the proofs in Section~\ref{sec4}.

\begin{corl}\label{corl2_3}
Let $n\ge 4$.
Let $\varphi\colon A[\tilde A_n]\to A[\tilde A_n]$ be an endomorphism.
\begin{itemize}
\item[(1)]
$\varphi$ is injective if and only if $\varphi$ is conjugate to an element of $\Aut^*(\tilde A_n)$.
\item[(2)]
$\varphi$ is surjective if and only if $\varphi$ is conjugate to an element of $\Aut^*(\tilde A_n)$.
\end{itemize}
\end{corl}

\begin{proof}
Let $\varphi\colon A[\tilde A_n]\to A[\tilde A_n]$ be an endomorphism.
By Theorem~\ref{thm2_1} we have one of the following four possibilities up to conjugation.
\begin{itemize}
\item[(1)]
$\varphi$ is cyclic.
\item[(2)]
$\varphi\in\Aut^*(\tilde A_n)$.
\item[(3)]
There exist $p\in\Z$ and $\psi\in\Aut^*(\tilde A_n)$ such that $\varphi=\psi\circ\alpha_p$.
\item[(4)]
There exist $p\in\Z$ and $\psi\in\Aut^*(\tilde A_n)$ such that $\varphi=\psi\circ\beta_p$.
\end{itemize}
Clearly, if $\varphi\in\Aut^*(\tilde A_n)$, then $\varphi$ is both injective and surjective.
It remains to show that $\varphi$ is neither injective nor surjective in the other three cases.

Suppose $\varphi$ is cyclic.
Then $\varphi$ is not injective because $\varphi(t_1)=\varphi(t_2)$.
It is not surjective  either because $A[\tilde A_n]$ is not cyclic.

Let $p\in\Z$ and $\psi\in\Aut^*(\tilde A_n)$.
Recall that $v_0=t_1\ldots t_{n-1}t_nt_{n-1}^{-1}\ldots t_1^{-1}$
and $\omega\colon A[\tilde A_n]\to W[\tilde A_n]$ denotes the epimorphism which sends $t_i$ to $t_i$ for all $0\le i\le n$.

We claim that $\omega(t_0)\notin W_Y[\tilde A_n]$. Indeed, by~\cite[Chapter 4, Section 1.8, Corollary~1]{Bourb1}, the subgroup $W_Y[\tilde A_n]$ consists of all elements of $W[\tilde A_n]$ for which there is a reduced expression in generators $S$ involving only elements of $Y\subset S$. Since $\omega(t_0)\notin Y$, we see that $\omega(t_0)\notin W_Y[\tilde A_n]$.

On the other hand, $\omega(v_0)\in W_Y[\tilde A_n]$, which means that $t_0\ne v_0$. However, we have $\alpha_p(t_0)=\alpha_p(v_0)=v_0\Delta_Y^{2p}$, hence $\alpha_p$ is not injective, and therefore $\psi\circ\alpha_p$ is not injective.

Similarly, $\Im(\alpha_p)\subset A_Y[\tilde A_n]$ and $t_0\notin A_Y[\tilde A_n]$ imply that
$\alpha_p$ is not surjective, and therefore $\psi\circ\alpha_p$ is not surjective.

We prove in the same way that $\psi\circ\beta_p$ is neither injective nor surjective.
\end{proof}

A first consequence of Corollary~\ref{corl2_3} is the determination of the automorphism group and of the outer automorphism group of $A[\tilde A_n]$ for $n\ge 4$.
Note that this result is already known and proved in~\cite{ChaCri1} for $n\ge 2$.

\begin{corl}[Charney--Crisp~\cite{ChaCri1}]\label{corl2_4}
Let $n\ge 4$.
Then 
\[
\Aut(A[\tilde A_n])=\Inn(A[\tilde A_n])\rtimes\Aut^*(\tilde A_n)\simeq A[\tilde A_n]\rtimes (D_{2(n+1)}\times
\Z/2\Z)\,,
\]
and
\[
\Out(A[\tilde A_n])=\Aut^*(\tilde A_n)\simeq D_{2(n+1)}\times\Z/2\Z\,,
\]
where $D_{2(n+1)}$ denotes the dihedral group of order $2(n+1)$.\qed
\end{corl}

Recall that a group $G$ is called \emph{Hopfian} if every surjective homomorphism $\psi\colon G\to G$ is injective, and it is called \emph{co-Hopfian} if every injective homomorphism $\psi\colon G\to G$ is surjective.
Another straightforward consequence of Corollary~\ref{corl2_3} is the following.

\begin{corl}\label{corl2_5}
Let $n\ge 4$.
Then $A[\tilde A_n]$ is Hopfian and co-Hopfian.\qed
\end{corl}

This result is also known. That $A[\tilde A_n]$ is co-Hopfian was established in~\cite[Corollary~3]{BelMar2} for $n\ge2$, and that $A[\tilde A_n]$ is Hopfian follows from the fact that it is a finitely generated residually finite group (as a subgroup of a linear group $A[B_{n+1}]$, see Section~\ref{sec3}).

\begin{rem}
The proofs of Corollaries~\ref{corl2_4}, \ref{corl2_5} given in~\cite{ChaCri1} and \cite{BelMar2}, respectively, rely on works of Ivanov~\cite{Ivano1} and Korkmaz~\cite{Korkm1} on the automorphism group of the mapping class group of a surface, and these techniques cannot be used to determine the whole set of endomorphisms of $A[\tilde A_n]$.
Our proof of Theorem~\ref{thm2_1} uses other tools, notably the study of the homomorphisms from $A[A_n]$ to $A[A_{n+1}]$ made in~\cite{Caste1} and~\cite{ChKoMa1}.
\end{rem}

%%%%%%%%%%

\section{Preliminaries}\label{sec3}

\subsection{Key embeddings}

The first ingredient in the proof of Theorem~\ref{thm2_1} is the following well-known chain of embeddings
\[
A[A_n] \lhook\joinrel\longrightarrow A[\tilde A_n] \lhook\joinrel\longrightarrow A[B_{n+1}] \lhook\joinrel\longrightarrow A[A_{n+1}]\,,
\]
which is described as follows.

The first embedding $\iota_Y\colon A[A_n]\lhook\joinrel\longrightarrow A[\tilde A_n]$ is defined by $\iota_Y(s_i) = t_i$ for all $1 \le i \le n$.

Now, we describe the second embedding, $\iota_{\tilde A}\colon A[\tilde A_n] \lhook\joinrel\longrightarrow A[B_{n+1}]$.
Let $\rho_B = r_1 \ldots r_n r_{n+1} \in A[B_{n+1}]$.
A direct calculation shows that $\rho_B r_i \rho_B^{-1} = r_{i+1}$ for all $1\le i\le n-1$ and $\rho_B^2  r_n \rho_B^{-2} = r_1$.
So if we set $r_0=\rho_B r_n\rho_B^{-1}$, then $\rho_Br_i\rho_B^{-1}=r_{i+1}$ for all $0\le i\le n-1$ and $\rho_Br_n\rho_B^{-1}=r_0$. (We note, however, that $r_{n+1}$ and $r_0$ are two different elements of $A[B_{n+1}]$, since their images in the abelianization of $A[B_{n+1}]$ are distinct.)
On the other hand, recall the automorphism $\zeta \in \Aut(\tilde A_n)$ which sends $t_i$ to $t_{i+1}$ for all $0 \le i \le n$.
We consider the action of an infinite cyclic group $\langle u \rangle \simeq \Z$ on $A[\tilde A_n]$ defined by $u \cdot h = \zeta(h)$ for all $h \in A[\tilde A_n]$, and we consider the semi-direct product $A[\tilde A_n] \rtimes \langle u \rangle$ defined by this action. 
The following is proved in~\cite{KenPei1} (see also~\cite{ChaCri1}).

\begin{prop}[Kent--Peifer~\cite{KenPei1}]\label{prop3_1}
There exists an isomorphism $\gamma\colon A[\tilde A_n] \rtimes \langle u \rangle \to A[B_{n+1}]$ which sends $t_i$ to $r_i$ for all $0\le i \le n$ and sends $u$ to $\rho_B$.\qed
\end{prop}

The embedding $\iota_{\tilde A}\colon A[\tilde A_n] \lhook\joinrel\longrightarrow A[B_{n+1}]$ is defined to be the restriction of $\gamma$ to $A[\tilde A_n]$.
From now on we identify $A[\tilde A_n]$ with the image of $\iota_{\tilde A}$.
So, $A[\tilde A_n]$ is viewed as the subgroup of $A[B_{n+1}]$ generated by $r_0, r_1, \dots, r_n$, and we have $t_i = r_i$ for all $0 \le i \le n$ via this identification.

\begin{rem}
We have a homomorphism $z\colon A[B_{n+1}] \to \Z$ defined by $z(r_i) = 0$ for all $1 \le i \le n$ and $z(r_{n+1}) = 1$.
As $z(\rho_B) = 1$, we have $\Ker(z) = \Im(\iota_{\tilde A}) = A[\tilde A_n]$.
\end{rem}

In what follows it will be useful to differentiate the generators of $A[A_n]$ from those of $A[A_{n+1}]$, hence we will denote by $s_1', \dots, s_n', s_{n+1}'$ the standard generators of $A[A_{n+1}]$.
The following proposition has been observed independently by several authors (see~\cite{Lambr1,Crisp2}, for example) 
and it is implicitly proved in~\cite{Bries1}. A simple combinatorial proof of it is given in~\cite[Proposition~1]{Manfr1}.

\begin{prop}\label{prop3_2}
There exists an injective homomorphism $\iota_B\colon A[B_{n+1}] \lhook\joinrel\longrightarrow A[A_{n+1}]$ which sends $r_i$ to $s_i'$ for all $1\le i\le n$ and sends $r_{n+1}$ to $(s_{n+1}')^2$.\qed
\end{prop}

From now on we identify $A[B_{n+1}]$ with the image of $\iota_B$.
So, $A[B_{n+1}]$ is viewed as the subgroup of $A[A_{n+1}]$ generated by $s_1',\dots,s_n',(s_{n+1}')^2$, and we have $r_i=s_i'$ for all $1\le i\le n$ and $r_{n+1}=(s_{n+1}')^2$ via this identification.

\begin{rem}\label{rem4}
Recall that $\omega\colon A[A_{n+1}] \to W[A_{n+1}]$ denotes the epimorphism which sends $s_i'$ to $s_i'$ for all $1 \le i\le n+1$. Then $A[B_{n+1}]=\omega^{-1}(W_Y[A_{n+1}])$, where $Y=\{s_1', \dots, s_n'\}$. This is proved in~\cite[Proposition~1]{Manfr1}, but also can be seen directly as follows. Observe that $A[B_{n+1}]$ maps onto $W_Y[A_{n+1}]$ under $\omega$ (since $\omega(s'^2_{n+1})=1$), and that $A[B_{n+1}]$ contains $\Ker\,\omega$. Indeed, $A[B_{n+1}]$ contains all normal generators $s'^2_1,\dots,s'^2_{n+1}$ of $\Ker\,\omega$ and all their conjugates by the generators $(s'_i)^{\pm1}$ of $A[A_{n+1}]$. This is clear for $(s'_i)^{\pm1}(s'_j)^2(s'_i)^{\mp1}$ if $(i,j)\ne(n+1,n)$. For the remaining case we see that: $s'_{n+1}s'^2_{n}(s'_{n+1})^{-1}=(s'_{n+1}s'_{n}(s'_{n+1})^{-1})^2=((s'_{n})^{-1}s'_{n+1}s'_{n})^2=(s'_{n})^{-1}s'^2_{n+1}s'_{n}\in A[B_{n+1}]$, and similarly for the conjugation by $s'^{-1}_{n+1}$. Thus, $A[B_{n+1}]$ contains $\Ker\,\omega$, which finishes the proof.
\end{rem}

The second ingredient in the proof of Theorem~\ref{thm2_1} is the following.

\begin{thm}[Castel~\protect{\cite[Th\'eor\`eme\,4\,(iii)]{Caste1}}, Chen--Kordek--Margalit~\protect{\cite[Corollary~1.2]{ChKoMa1}}, Orevkov~\protect{\cite[Theorem~1.6]{Orevk1}}]\label{thm3_3}
Let $n\ge 4$.
Let $\varphi\colon A[A_n]\to A[A_{n+1}]$ be a homomorphism.
Then we have one of the following two possibilities up to conjugation.
\begin{itemize}
\item[(1)]
$\varphi$ is cyclic.
\item[(2)]
There exist $\varepsilon\in\{\pm 1\}$ and $p,q\in\Z$ such that $\varphi(s_i)=s_i'^\varepsilon \Delta_Y^{2p}\Delta^{2q}$, where $Y=\{s_1',\dots,s_n'\}$, $\Delta_Y=\Delta_Y[A_{n+1}]$ and $\Delta=\Delta[A_{n+1}]$.\qed
\end{itemize}
\end{thm}

Our strategy for proving Theorem~\ref{thm2_1} will be now as follows. 
Firstly, we use Theorem~\ref{thm3_3} to determine the set of homomorphisms from $A[\tilde A_n]$ to $A[A_{n+1}]$ (see Proposition~\ref{prop4_1}).
Secondly, we determine which of these homomorphisms have an image contained in $A[\tilde A_n]\subset A[A_{n+1}]$.

\subsection{Mapping class groups}

To pass from homomorphisms from $A[A_n]$ to $A[A_{n+1}]$ to homomorphisms from $A[\tilde A_n]$ to $A[A_{n+1}]$, we use a third ingredient: the mapping class groups.
So, we give below the information on these groups that we will need, and we refer to~\cite{FarMar1} for a complete exposition on the subject.

Let $\Sigma$ be a compact oriented surface with or without boundary, and let $\PP$ be a finite family of marked points in the interior of $\Sigma$.
We denote by $\Homeo^+(\Sigma,\PP)$ the group of homeomorphisms of $\Sigma$ which preserve the orientation, which are the identity on a neighborhood of the boundary, and which leave set-wise invariant the set $\PP$.
The \emph{mapping class group} of the pair $(\Sigma,\PP)$, denoted by $\MM(\Sigma,\PP)$, is the group of isotopy classes of elements of $\Homeo^+(\Sigma,\PP)$, where isotopies are required to leave the set $\partial \Sigma \cup \PP$ point-wise invariant. 
If $\PP=\varnothing$, then we write $\Homeo^+(\Sigma)=\Homeo^+(\Sigma,\varnothing)$ and $\MM(\Sigma)=\MM(\Sigma,\varnothing)$.

A \emph{circle} of $(\Sigma,\PP)$ is the image of an embedding $a\colon \S^1 \hookrightarrow \Sigma \setminus (\partial\Sigma \cup \PP)$.
It is called \emph{generic} if it does not bound any disk containing $0$ or $1$ marked point, and if it is not parallel to any boundary component.
The isotopy class of a circle $a$ is denoted by $[a]$. We emphasize that isotopies are considered in $\Sigma \setminus (\partial\Sigma \cup \PP)$, i.e.\ circles are not allowed to pass through points of $\partial\Sigma \cup\PP$ under isotopies. 
We denote by $\CC(\Sigma,\PP)$ the set of isotopy classes of generic circles.
The \emph{intersection index} of two classes $[a],[b]\in\CC(\Sigma,\PP)$ is $i([a],[b])=\min\{|a'\cap b'|\mid a'\in[a]\text{ and }b'\in[b]\}$.
The set $\CC(\Sigma,\PP)$ is endowed with a structure of simplicial complex, where a non-empty finite subset $\FF\subset\CC(\Sigma,\PP)$ is a simplex if $i([a],[b])=0$ for all $[a],[b]\in\FF$.
This complex is called the \emph{curve complex} of $(\Sigma,\PP)$.

In the present paper the (right) Dehn twist along a circle $a$ is denoted by $T_a$.
The following is important to make the link between Artin groups and mapping class groups.
Its proof can be found in~\cite[Section 3.5]{FarMar1}.

\begin{prop}\label{prop3_4}
Let $\Sigma$ be an oriented compact surface and let $\PP$ be a finite collection of marked points in the interior of $\Sigma$.
Let $a,b$ be two generic circles of $(\Sigma,\PP)$ which we assume to be in minimal position, that is, such that $i([a],[b])=|a\cap b|$.
Then for the Dehn twists $T_a$, $T_b$ we have:
\begin{itemize}
\item[(1)]
%We have 
$T_aT_b=T_bT_a$, if $a\cap b=\varnothing$.
\item[(2)]
%We have 
$T_aT_bT_a=T_bT_aT_b$, if $|a\cap b|=1$.
\item[(3)]
The set $\{T_a,T_b\}$ generates a rank~$2$ free group, if $|a\cap b|\ge 2$.\qed
\end{itemize}
\end{prop}

An \emph{arc} of $(\Sigma,\PP)$ is the image of an embedding $a\colon [0,1]\hookrightarrow\Sigma\setminus\partial\Sigma$ such that $a([0,1])\cap\PP=\{a(0),a(1)\}$, with $a(0)\ne a(1)$. We consider arcs up to isotopies under which interior points of arcs are mapped into $\Sigma\setminus(\partial\Sigma\cup\PP)$.
We denote the isotopy class of an arc $a$ by $[a]$, and the set of isotopy classes of arcs by $\AA(\Sigma,\PP)$.
The \emph{intersection index} of two classes $[a],[b]\in\AA(\Sigma,\PP)$ is $i([a],[b])=\min\{|a'\cap b'|\mid a'\in[a]\text{ and }b'\in[b]\}$.

With an arc $a$ of $(\Sigma,\PP)$ we associate an element $H_a\in\MM(\Sigma,\PP)$, called the \emph{(right) half-twist} along $a$.
This element is the identity outside a regular neighborhood of $a$, it exchanges the two ends of $a$, and $H_a^2=T_{\hat a}$, where $\hat a$ is the boundary of a regular neighborhood of $a$, see Figure~\ref{fig:halftwist}. We refer the reader to~\cite[Section~1.6.2]{KasTur1} for more information about properties of half-twists.

\begin{figure}[htb!]
\begin{center}
\begin{tikzpicture}[scale=0.8]
%\draw [help lines] (-2,-1) grid (3,2);
\draw[thick] (0,0) ellipse [x radius=2, y radius=1];
\draw[dotted, very thick, blue] (-2,0)--(-1,0);
\draw[dotted, very thick, blue] (1,0)--(2,0);
\draw[->-=0.55,thick, blue] (-1,0)--(1,0);
\draw (0,0.3) node {\footnotesize$a$};
\draw (2,0.7) node {\footnotesize$\hat a$};
\fill (-1,0) circle (2pt);
\fill (1,0) circle (2pt);

\begin{scope}[xshift=7cm]
\draw[thick] (0,0) ellipse [x radius=2, y radius=1];
\draw[dotted, very thick, blue] (2,0) to [out=135,in=45,looseness=0.75] (-1,0); 
\draw[dotted, very thick, blue] (-2,0) to [out=-45,in=-135,looseness=0.75] (1,0);
\draw[->-=0.55,thick, blue] (1,0)--(-1,0);
\fill (-1,0) circle (2pt);
\fill (1,0) circle (2pt);
\end{scope}

\draw[->] (2.75,0) to (4.25	,0); 
\draw (3.5,0.3) node {$H_a$};

\end{tikzpicture}
\end{center}
\caption{A half-twist\label{fig:halftwist}}
\end{figure}

The following is known to experts but, as far as we know, the proof is not readily available in the literature, so we provide it here.

\begin{prop}\label{prop3_5}
Let $\Sigma$ be an oriented compact surface and let $\PP$ be a finite collection of marked points in the interior of $\Sigma$. If $|\PP|=3$, or if $|\PP|=4$ and the genus of $\Sigma$ is zero, we require additionally that $\partial\Sigma\ne\varnothing$. Let $a,b$ be two non-isotopic arcs of $(\Sigma,\PP)$ which we assume to be in minimal position, that is, such that $i([a],[b])=|a\cap b|$. Then for the half-twists $H_a$ and $H_b$ we have:
\begin{itemize}
\item[(1)]
%We have 
$H_aH_b=H_bH_a$, if $a\cap b=\varnothing$.
\item[(2)]
%We have 
$H_aH_bH_a=H_bH_aH_b$, if $|a\cap b|=1$ and $a\cap b\subset\PP$.
\item[(3)]
The set $\{H_a,H_b\}$ generates a rank~$2$ free group, if $|a\cap b|\ge 2$, or if $|a\cap b|=1$ and $a\cap b\not\subset\PP$.
\end{itemize}
\end{prop}

Proposition~\ref{prop3_5} will follow from the Birman--Hilden theory combined with Proposition~\ref{prop3_4}.
So, we begin by recalling the results of Birman--Hilden~\cite{BirHil1} that we need and we refer to~\cite{MarWin1} or~\cite[Chapter 9]{FarMar1} for more details.

As above, let $\Sigma$ be a compact oriented surface with or without boundary, and let $\PP$ be a finite family of marked points in the interior of $\Sigma$. 
Let $p\colon \tilde\Sigma\to\Sigma$ be a two-sheeted branched covering whose set of branch points is exactly $\PP$, and let $\tilde\PP=p^{-1}(\PP)$. In what follows we assume that the Euler characteristic of $\tilde\Sigma$ is negative.

Let $\sigma\colon\tilde \Sigma \to \tilde \Sigma$ be the non-trivial deck transformation of $p$. 
Note that $\sigma$ is not necessarily an element of $\Homeo^+ (\tilde \Sigma, \tilde \PP)$ in the sense that it does not necessarily restrict to the identity on a neighborhood of the boundary of $\tilde \Sigma$. 
In fact, we have $\sigma \in \Homeo^+ (\tilde \Sigma, \tilde \PP)$ if and only if $\partial \Sigma = \varnothing$. 
Let $\SHomeo^+ (\tilde \Sigma, \tilde \PP)$ denote the subgroup of $\Homeo^+  (\tilde \Sigma, \tilde \PP)$ consisting of the elements $\tilde f \in \Homeo^+ (\tilde \Sigma, \tilde \PP)$ which commute with $\sigma$, and let $\SS\MM (\tilde \Sigma, \tilde \PP)$ denote the subgroup of $\MM (\tilde \Sigma)$ consisting of the isotopy classes of elements of $\SHomeo^+ (\tilde \Sigma, \tilde \PP)$. We remark that the isotopies in the definition of $\MM(\tilde\Sigma)$ are not required to preserve the set $\tilde\PP$.

Let $f \in \Homeo^+ (\Sigma, \PP)$. 
We say that $f$ is \emph{liftable} if there exists $\tilde f \in \Homeo^+ (\tilde \Sigma, \tilde \PP)$ such that $p \circ \tilde f = f \circ p$. 
In this case $\tilde f$ is an element of $\SHomeo^+ ( \tilde \Sigma, \tilde \PP)$. 
Note that, if $\partial \tilde \Sigma \neq \varnothing$, then $\tilde f$ is unique, and if $\partial \tilde \Sigma = \varnothing$, then $f$ has exactly two lifts, $\tilde f$ and $\tilde f \circ \sigma = \sigma \circ \tilde f$.  
We denote by $\LHomeo^+ (\Sigma, \PP)$ the subgroup of liftable elements of $\Homeo^+ (\Sigma, \PP)$, and we denote by $\LL \MM (\Sigma, \PP)$ the subgroup of $\MM (\Sigma, \PP)$ consisting of the isotopy classes of elements of $\LHomeo^+ (\Sigma, \PP)$.

We have a surjective homomorphism $\Phi\colon\SHomeo^+ (\tilde \Sigma, \tilde \PP) \to \LHomeo^+ (\Sigma, \PP)$ defined as follows.
Let $\tilde f \in \SHomeo^+ (\tilde \Sigma, \tilde \PP)$.
Then $\Phi (\tilde f)$ is the unique homeomorphism $f \in \LHomeo^+ (\Sigma, \PP)$ satisfying $p \circ \tilde f = f \circ p$. 
We have $\Ker (\Phi) = G$, where $G = \langle \sigma \rangle = \{ \id, \sigma \}$ if $\partial \Sigma = \varnothing$, and $G = \{ \id \}$ if $\partial \Sigma \neq \varnothing$.
Since the degree two covers are regular, and the Euler characteristic of $\tilde\Sigma$ is negative, we can apply results of \cite[Theorem 2]{BirHil1} and \cite[Section 3]{MarWin1} and conclude that $\Phi$ induces a surjective homomorphism $\Phi^*\colon\SS \MM (\tilde \Sigma, \tilde \PP) \to \LL \MM (\Sigma, \PP)$ whose kernel is $G'$, where $G' = \langle [ \sigma ] \rangle = \{ \id, [\sigma]\} \simeq \Z/ 2 \Z$ if $\partial \Sigma = \varnothing$, and $G' = \{ \id \}$ if $\partial \Sigma \neq \varnothing$.

\begin{proof}[Proof of Proposition~\ref{prop3_5}]
We would like to engineer a surface $\tilde\Sigma$ and a degree two branched cover $\tilde\Sigma\to \Sigma$ such that the set of branched points is exactly $\PP$. This is not always possible, e.g.\ in the case when the cardinality of $\PP$ is odd and the surface $\Sigma$ has no boundary. So we first show how to modify the surface $\Sigma$ in order to reduce the problem to the case of even number of marked points, and then how to ensure that the Euler characteristic of $\tilde\Sigma$ is negative.

\emph{Step 1: Reduction to the case of even number of marked points.} Since arcs $a,b$ are assumed to exist in $(\Sigma,\PP)$, the cardinality of $\PP$ is at least $2$. Suppose that $\card\PP$ is odd. 

Assume first that there exist a marked point $x\in \PP\setminus(a\cup b)$. Consider a surface $\Sigma'$ which is obtained from $\Sigma$ by removing the point $x$ and ``blowing up'' the resulting hole into a boundary circle $c_x$. Thus, there exist a quotient map of surfaces $q\colon \Sigma'\to\Sigma$, which sends the boundary circle $c_x\subset\Sigma'$ to the point $x\in\Sigma$. Let $a'$, $b'$ be the arcs $a$, $b$, viewed on the surface $\Sigma'$. Clearly, $a'$ and $b'$ lie entirely in the interior of $\Sigma'$ and they are also in the minimal position in $\Sigma'$. We notice that the subgroup $\langle H_a, H_b\rangle$ lies in the subgroup $\MM(\Sigma, x)$ of mapping classes that fix the marked point $x$. And it is known (see~\cite[Proposition~3.19]{FarMar1}) that the quotient map $q$ induces the ``capping epimorphism'' $q^*\colon\MM(\Sigma')\to\MM(\Sigma,x)$ with the kernel generated by the boundary Dehn twist $T_{c_x}$. In particular, $q^*(H_{a'})=H_a$, $q^*(H_{b'})=H_b$, and $H_a$, $H_b$ satisfy relations $H_aH_b=H_bH_a$ or $H_aH_bH_a=H_bH_aH_b$, if elements $H_{a'}$, $H_{b'}$ satisfy these relations in $\MM(\Sigma')$. Moreover, since the boundary twist $T_{c_x}$ is central in $\MM(\Sigma')$, the elements $H_a$, $H_b$ generate a rank 2 free subgroup in $\MM(\Sigma,x)$, if $H_{a'}$, $H_{b'}$ generate a rank 2 free subgroup in $\MM(\Sigma')$. This allows us to consider the surface $(\Sigma',\PP')$ instead of $(\Sigma,\PP)$, where $\PP'=\PP\setminus \{x\}$ has even cardinality.

Now we consider the case when $\card\PP$ is odd but there is no marked point $x\in\PP\setminus(a\cup b)$. Then $\card\PP=3$, and by the assumption of the proposition, $\Sigma$ has a non-empty boundary.
Let $c$ be a boundary component of $\Sigma$. 
We glue an annulus $A$ to $c$ and we take a new marked point $x$ inside this annulus. 
We denote by $\Sigma'$ the surface obtained in this way and we set $\PP' = \PP \cup \{x\}$. 
Then $a$ and $b$ are still in minimal position in $(\Sigma', \PP')$ and the embedding of $\Sigma$ into $\Sigma'$ induces an injective homomorphism $\MM (\Sigma, \PP) \lhook\joinrel\longrightarrow \MM (\Sigma', \PP')$, see~\cite[Theorem~3.18]{FarMar1}. In particular, all the conclusions concerning $H_a$, $H_b$ in parts (1), (2), (3) of the proposition are satisfied in $\MM(\Sigma,\PP)$ if and only if they are satisfied in $\MM(\Sigma',\PP')$. Therefore, we can assume without loss of generality, that $\card\PP$ is even.

\emph{Step 2: Ensuring that the Euler characteristic of $\tilde\Sigma$ is negative.} The Birman--Hilden theorem cited above applies only to the covering surfaces of negative Euler characteristic. In the construction of Step~3 below, the genus of the resulting covering surface $\tilde\Sigma$ will be equal to $2\genus(\Sigma)+\frac12\card\PP-1$. Thus, to ensure that the Euler characteristic of $\tilde\Sigma$ is negative, it is sufficient to modify the surface $\Sigma$ in such a way that $\genus(\Sigma)\ge1$ if $\card\PP\le4$.

Suppose that $\genus(\Sigma)=0$. Since, by the assumption, at least two non-isotopic arcs exist in $(\Sigma,\PP)$, we see that if $\card\PP=2$, then $\Sigma$ is not a closed sphere or a disk. In particular, $\partial\Sigma\ne\varnothing$. If $\card\PP=4$ then, by the assumption of the proposition, again, $\partial\Sigma\ne\varnothing$. So let $c$ be a boundary component of $\Sigma$. Consider a compact oriented surface $S$ of genus $1$ with one boundary component (a torus with one disk removed). We glue the surface $S$ to $c$ and denote by $\Sigma'$ the resulting surface. Again, we notice that $a$ and $b$ are still in minimal position in $(\Sigma', \PP)$ and the embedding of $\Sigma$ into $\Sigma'$ induces an injective homomorphism $\MM (\Sigma, \PP) \lhook\joinrel\longrightarrow \MM (\Sigma', \PP)$, by~\cite[Theorem~3.18]{FarMar1}. Therefore, we can assume without loss of generality, that $\genus(\Sigma)\ge1$ if $\card\PP\le4$.

\emph{Step 3: Construction of the cover $\tilde\Sigma$.}
Let $p\colon \tilde\Sigma\to\Sigma$ be the two-sheeted branched covering whose set of branch points is precisely $\PP$.
To describe $p$ explicitly, we choose an embedded disk $D\subset\Sigma$ such that $\PP$ lies in the interior of $D$ and denote $\Sigma_0=(\Sigma\setminus D)\cup \partial D$. Then $\tilde\Sigma$ is built as the union of two copies of $\Sigma_0$ attached to a surface $\Sigma_D$ which has genus $\frac12\card\PP-1$ and two boundary components. The branched covering $p$ restricted to $\Sigma_D$ is the quotient map $\Sigma_D\to D$ by the action of the hyperelliptic involution on $\Sigma_D$, and $p$ sends each copy of $\Sigma_0\subset \tilde \Sigma$ homeomorphically onto $\Sigma_0\subset \Sigma$. Notice that $\genus(\tilde\Sigma)=2\genus(\Sigma)+\frac12\card\PP-1$.

Set $\tilde \PP= p^{-1}(\PP)$ and let $\Phi\colon\SHomeo^+ (\tilde \Sigma, \tilde \PP) \to \LHomeo^+ (\Sigma, \PP)$ and $\Phi^*\colon\SS \MM (\tilde \Sigma, \tilde \PP) \to \LL \MM (\Sigma, \PP)$ be the surjective homomorphisms defined as described before the proof of the proposition.

\emph{Step 4: Relating half-twists in $(\Sigma,\PP)$ to Dehn twists in $\tilde\Sigma$.}
We observe that the full preimage $p^{-1}(a)$ is a circle $\tilde a$ in $\tilde \Sigma$. Since $\genus(\Sigma)\ge1$ or $\card\PP\ge6$, we see from the construction in Step~3 that $\tilde a$ does not bound a disk in $\tilde \Sigma$, and neither is $\tilde a$ parallel to a boundary component of $\tilde\Sigma$. We conclude that $\tilde a$ is a generic circle in $\tilde \Sigma$, with $H_a \in \LL \MM (\Sigma, \PP)$, $T_{\tilde a} \in \SS \MM (\tilde \Sigma, \tilde \PP)$, and $\Phi_* (T_{\tilde a}) = H_a$ (see~\cite[Section~9.4.1]{FarMar1}).
Similarly, $\tilde b = p^{-1} (b)$ is a generic circle with $H_b \in \LL \MM (\Sigma, \PP)$, $T_{\tilde b} \in \SS \MM (\tilde \Sigma, \tilde \PP)$, and $\Phi_* (T_{\tilde b}) = H_b$.

Arguing as in the proof of~\cite[Lemma~9.3]{FarMar1}, we see that the circles $\tilde a=p^{-1}(a)$ and $\tilde b=p^{-1}(b)$ are in minimal position. If $a\cap b=\varnothing$, then $\tilde a\cap\tilde b=\varnothing$, hence, by Proposition~\ref{prop3_4}, $T_{\tilde a}T_{\tilde b}=T_{\tilde b}T_{\tilde a}$, and therefore, by applying $\Phi_*$ we get $H_aH_b=H_bH_a$.
If $|a\cap b|=1$ and $a\cap b\subset\PP$, then $|\tilde a\cap\tilde b|=1$, hence, by Proposition~\ref{prop3_4}, $T_{\tilde a}T_{\tilde b}T_{\tilde a}=T_{\tilde b}T_{\tilde a}T_{\tilde b}$, and therefore, by applying $\Phi_*$ we get $H_aH_bH_a=H_bH_aH_b$.
Suppose that $|a\cap b|\ge2$, or $|a\cap b|=1$ and $a\cap b\not\subset\PP$.
Then clearly $|\tilde a\cap\tilde b|\ge 2$, and since $\tilde a$ and $\tilde b$ are in minimal position, $i([\tilde a], [\tilde b])\ge 2$. Hence, by Proposition~\ref{prop3_4}, $\{T_{\tilde a},T_{\tilde b}\}$ generates a rank~$2$ free group.
Note that $\langle T_{\tilde a},T_{\tilde b}\rangle\cap G'=\{\id\}$, because $\langle T_{\tilde a},T_{\tilde b}\rangle$ is torsion-free and $G'$ is finite.
Since $G'=\Ker(\Phi_*)$, it follows that the restriction of $\Phi_*$ to $\langle T_{\tilde a},T_{\tilde b}\rangle$ is injective, hence $\{\Phi_*(T_{\tilde a}),\Phi_*(T_{\tilde b})\}=\{H_a,H_b\}$ generates a rank~$2$ free group.
\end{proof}

\begin{rem}
\begin{enumerate}
\item In the proof of Proposition~\ref{prop3_5}, if one wishes to use the original version of the Birman--Hilden theorem for closed surfaces, one can ensure that $\tilde\Sigma$ has no boundary by gluing in a torus with one boundary component to every boundary component of $\Sigma$, in the way as it is done in Step 2 of the proof for the case $\genus(\Sigma)=0$ and $\card\PP\le4$.
\item 
We believe that the conclusions of Proposition~\ref{prop3_5} 
are true in the special cases $|\PP|=3$, and $|\PP|=4$ with $\genus(\Sigma)=0$, even without the assumption $\partial\Sigma\ne\varnothing$.
\end{enumerate}
\end{rem}

%%ICI%%

We record an important corollary, which is an analog for half-twists of the corresponding fact about Dehn twists, stating that the two Dehn twists uniquely determine their respective circles up to isotopy (see~\cite[Fact~3.6]{FarMar1}).
\begin{corl}\label{cor3_6}
Let $(\Sigma,\PP)$ be as in Proposition~\ref{prop3_5}. If $a$ and $b$ are two non-isotopic arcs in $(\Sigma,\PP)$, then their respective half-twists $H_a$ and $H_b$ are different elements of $\MM(\Sigma,\PP)$.
\end{corl}
\begin{proof}
Assume that $H_a=H_b$ in $\MM(\Sigma,\PP)$. 
Observe that $H_a$ interchanges points $a(0)$ and $a(1)$ and fixes $\PP\setminus\{a(0),a(1)\}$ pointwise, and, similarly, $H_b$ interchanges points $b(0)$, $b(1)$ and fixes $\PP\setminus\{b(0),b(1)\}$ pointwise. Since isotopies in the definition of $\MM(\Sigma,\PP)$ are required to fix the set $\PP$ pointwise, we conclude that the two sets are equal: $\{a(0),a(1)\}=\{b(0),b(1)\}$. But then, by part (3) of Proposition~\ref{prop3_5}, $H_a$ and $H_b$ generate a rank $2$ free group in $\MM(\Sigma,\PP)$, which is a contradiction with the assumption $H_a=H_b$. 
\end{proof}

We denote by $\D$ the standard disk, and we choose a collection $\PP_{n+2}=\{p_0,p_1,\dots, p_n,p_{n+1}\}$ of $n+2$ marked points in the interior of $\D$ (see Figure~\ref{fig3_1}).
Let $a_1,\dots,a_n,a_{n+1}$ be the arcs drawn in Figure~\ref{fig3_1}.
Then, by Proposition~\ref{prop3_5}, we have a homomorphism $\Psi\colon A[A_{n+1}]\to\MM(\D,\PP_{n+2})$ which sends $s_i'$ to $H_{a_i}$ for all $1\le i\le n+1$.
By~\cite[Theorem~1.33]{KasTur1}, this homomorphism is an isomorphism.
From now on we will identify $A[A_{n+1}]$ with $\MM(\D,\PP_{n+2})$ via $\Psi$.
In particular, we assume that $s_i'=H_{a_i}$ for all $1\le i\le n+1$.
We claim that, via this identification, we have $t_0=H_{b_0}$, where $b_0$ is the arc drawn in Figure~\ref{fig3_1}. Indeed, recall that $t_0\in A[\tilde A_n]$ is identified under the embedding $\iota_{\tilde A}$ with the element $r_0=\rho_Br_n\rho_B^{-1}$ of $A[B_{n+1}]$, where $\rho_B=r_1\ldots r_nr_{n+1}$. In its turn, under the embedding $\iota_B$, elements $r_n$ and $\rho_B$ get identified with $s_n'$ and $s_1's_2'\ldots s_n'(s_{n+1}')^2$, respectively. Hence, the image of $t_0$ in $\MM(\D,\PP_{n+2})$ under $\Psi$ is $\Psi(\rho_B)H_{a_n}\Psi(\rho_B)^{-1}=H_{\Psi(\rho_B)(a_n)}$. One easily checks by drawing pictures that $\Psi(\rho_B)(a_n)=H_{a_1}H_{a_2}\ldots H_{a_n}H_{a_{n+1}}^2(a_n)=b_0$, up to isotopy.

\begin{figure}[ht!]
\begin{center}
\includegraphics[width=4cm]{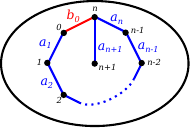}
\caption{A disk with marked points}\label{fig3_1}
\end{center}
\end{figure}

\subsection{Reduction systems of curves}

Let $\Sigma$ be an oriented compact surface, and let $\PP$ be a finite collection of marked points in the interior of $\Sigma$.
Assume the Euler characteristic of $\Sigma\setminus\PP$ is negative.
Let $f\in\MM(\Sigma,\PP)$.
We say that a simplex $\FF$ of $\CC(\Sigma,\PP)$ is a \emph{reduction system} for $f$ if $f(\FF)=\FF$.
In this case any element of $\FF$ is called a \emph{reduction class} for $f$.
A reduction class $[a]$ is an \emph{essential reduction class} if, for each $[b]\in\CC(\Sigma,\PP)$ such that $i([a],[b])\neq 0$ and for each $m\in\Z\setminus\{0\}$, we have $f^m([b])\neq [b]$.
In particular, if $[a]$ is an essential reduction class and $[b]$ is any reduction class, then $i([a],[b])=0$.
We denote by $\SS(f)$ the set of essential reduction classes for $f$.
The following gathers together some results on $\SS(f)$ that we will use in the proofs.

\begin{thm}[Birman--Lubotzky--McCarthy~\cite{BiLuMc1}]\label{thm3_6}
Let $\Sigma$ be an oriented compact surface, and let $\PP$ be a finite collection of marked points in the interior of $\Sigma$.
Assume the Euler characteristic of $\Sigma\setminus\PP$ is negative.
Let $f\in\MM(\Sigma,\PP)$.
\begin{itemize}
\item[(1)]
If $\SS(f)\neq\varnothing$, then $\SS(f)$ is a reduction system for $f$.
\item[(2)]
We have $\SS(f^n)=\SS(f)$ for all $n\in\Z\setminus\{0\}$.
\item[(3)]
We have $\SS(gfg^{-1})=g(\SS(f))$ for all $g\in\MM(\Sigma,\PP)$.\qed
\end{itemize}
\end{thm}

In addition to Theorem~\ref{thm3_6} we have the following which is well-known and which is a direct consequence of~\cite{BiLuMc1}.

\begin{prop}\label{prop3_7}
Let $\Sigma$ be an oriented compact surface, and let $\PP$ be a finite collection of marked points in the interior of $\Sigma$.
Assume the Euler characteristic of $\Sigma\setminus\PP$ is negative.
Let $f_0\in Z(\MM(\Sigma,\PP))$ be a central element of $\MM(\Sigma,\PP)$, let $\FF=\{[a_1],[a_2],\dots,[a_p]\}$ be a simplex of $\CC(\Sigma,\PP)$, and let $k_1,k_2,\dots,k_p$ be non-zero integers.
Let $g=T_{a_1}^{k_1}T_{a_2}^{k_2}\ldots T_{a_p}^{k_p}f_0$.
Then $\SS(g)=\FF$.\qed
\end{prop}

%%%%%%%%%%

\section{Proofs}\label{sec4}

As pointed out in Section~\ref{sec3}, we start by determining the homomorphisms from $A[\tilde A_n]$ to $A[A_{n+1}]$.
As before, we set $Y = \{s_1', \dots, s_n'\} = \{t_1, \dots, t_n\}$, $\Delta_Y = \Delta_Y[A_{n+1}] = \Delta_Y [\tilde A_n]$ and $\Delta = \Delta[A_{n+1}]$.
Consider the following elements of $A[\tilde A_n] \subset A[A_{n+1}]$.
\begin{gather*}
u_0 = t_0\,,\ u_1 = \Delta_Y^{-1} t_0 \Delta_Y\,,\ v_0 = \rho t_n \rho^{-1} = t_1 \ldots t_{n-1} t_n t_{n-1}^{-1} \ldots t_1^{-1}\,,\\
v_1 = \rho' t_n \rho'^{-1} = t_1^{-1} \ldots t_{n-1}^{-1} t_n t_{n-1} \ldots t_1\,,
\end{gather*}
where $\rho = t_1 t_2 \ldots t_n$ and $\rho' = t_1^{-1} t_2^{-1} \ldots t_n^{-1}$.
One can easily see that the above elements are identified under the isomorphism $\Psi\colon A[A_{n+1}]\to \MM(\D,\PP_{n+2})$ with the following half-twists: $u_0 = H_{b_0}$, $u_1 = H_{b_1}$, $v_0 = H_{c_0}$ and $v_1 = H_{c_1}$, where $b_0, b_1, c_0, c_1$ are the arcs drawn in Figure~\ref{fig4_1}. (This can be checked using the property $fH_af^{-1}=H_{f(a)}$ for an orientation-preserving mapping class $f$.)

\begin{figure}[ht!]
\begin{center}
\includegraphics[width=4cm]{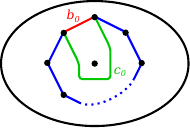}\qquad
\includegraphics[width=4cm]{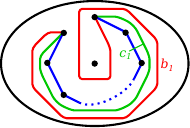}
\caption{Arcs in the disk with marked points}\label{fig4_1}
\end{center}
\end{figure}

\begin{prop}\label{prop4_1}
Let $n \ge 4$.
Let $\varphi\colon A[\tilde A_n] \to A[A_{n+1}]$ be a homomorphism.
Then we have one of the following three possibilities.
\begin{itemize}
\item[(1)]
$\varphi$ is cyclic.
\item[(2)]
There exist $g \in A[A_{n+1}]$, $k \in \{0,1\}$, $\varepsilon \in \{\pm 1\}$ and $q \in \Z$ such that $\varphi(t_i) = g s_i'^{\varepsilon} \Delta^{2q} g^{-1}$ for all $1 \le i \le n$, and $\varphi(t_0) = g u_k^ {\varepsilon} \Delta^{2q} g^{-1}$.
\item[(3)]
There exist $g \in A[A_{n+1}]$, $k \in \{0,1\}$, $\varepsilon \in \{\pm 1\}$ and $p, q \in \Z$ such that $\varphi(t_i) = g s_i'^{\varepsilon} \Delta_Y^{2p} \Delta^{2q} g^{-1}$ for all $1 \le i\le n$, and $\varphi(t_0) = g v_k^{\varepsilon} \Delta_Y^{2p} \Delta^{2q} g^{-1}$.
\end{itemize}
\end{prop}

\begin{rem}
It can be shown that a homomorphism $\varphi\colon A[\tilde A_n] \to A[A_{n+1}]$, $n\ge4$, is never surjective, and it is injective if and only if it belongs to Case~(2) of Proposition~\ref{prop4_1}. 
\end{rem}

\begin{proof}[Proof of Proposition~\ref{prop4_1}]
Let $\varphi\colon A[\tilde A_n] \to A[A_{n+1}]$ be a homomorphism.
Recall the embedding $\iota_Y\colon A[A_n] \hookrightarrow A[\tilde A_n]$ which sends $s_i$ to $t_i$ for all $1 \le i \le n$.
By Theorem~\ref{thm3_3} we have one of the following two possibilities.
\begin{itemize}
\item[(1)]
$\varphi \circ \iota_Y$ is cyclic.
\item[(2)]
There exist $g \in A[A_{n+1}]$, $\varepsilon \in \{\pm 1\}$, and $p, q \in \Z$ such that $(\varphi \circ \iota_Y) (s_i) = g s_i'^\varepsilon \Delta_Y^{2p} \Delta^{2q} g^{-1}$ for all $1 \le i \le n$.
\end{itemize}

Suppose $\varphi \circ \iota_Y$ is cyclic.
There exists $h \in A[A_{n+1}]$ such that $(\varphi \circ \iota_Y) (s_i) = \varphi(t_i) = h$ for all $1 \le i \le n$.
We also have
\begin{gather*}
\varphi(t_0) = \varphi (t_1 t_0 t_1 t_0^{-1} t_1^{-1}) = \varphi (t_1 t_0)\, \varphi(t_1)\, \varphi (t_0^{-1} t_1^{-1}) = \varphi (t_1 t_0)\, \varphi(t_3)\, \varphi (t_0^{-1} t_1^{-1}) = \\
\varphi (t_1 t_0 t_3 t_0^{-1} t_1^{-1}) = \varphi(t_3) = h\,,
\end{gather*}
hence $\varphi$ is cyclic.

Now, assume that there exist $g \in A[A_{n+1}]$, $\varepsilon \in \{\pm 1\}$, and $p, q \in \Z$ such that $( \varphi \circ \iota_Y) (s_i) = g s_i'^\varepsilon \Delta_Y^{2p} \Delta^{2q} g^{-1}$ for all $1 \le i\le n$.
From here the proof is divided into two cases depending on whether $p \neq 0$ or $p = 0$.

{\it Case 1: $p \neq 0$.}
Let $\varphi' = \conj_{g^{-1}} \circ \varphi$.
We have $\varphi'(t_i) = s_i'^\varepsilon \Delta_Y^{2p} \Delta^{2q}$ for all $1 \le i \le n$.
It is easily seen that $\Delta_Y^2 = T_d$, where $d$ is the circle drawn in Figure~\ref{fig4_2}, $\Delta^2 = T_{\partial \D}$, where $T_{\partial \D}$ denotes the Dehn twist along the boundary of $\D$, and $s_i'=H_{a_i}$, where $a_i$ is the arc drawn in Figure~\ref{fig3_1}, for all $1\le i\le n$.
So, $\varphi'(t_i) = H_{a_i}^\varepsilon T_d^p T_{\partial\D}^q$ for all $1 \le i \le n$.
For each $1\le i\le n$ we denote by $\hat a_i$ the boundary of a regular neighborhood of $a_i$.
Then $\varphi'(t_i)^2 = T_{\hat a_i}^\varepsilon T_d^{2p} T_{\partial\D}^{2q}$, hence, by Proposition~\ref{prop3_7}, $\SS (\varphi'(t_i)) = \SS (\varphi'(t_i)^2) = \{[\hat a_i], [d]\}$ for all $1 \le i \le n$.

\begin{figure}[ht!]
\begin{center}
\includegraphics[width=4cm]{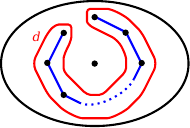}
\caption{A circle in the disk with marked points}\label{fig4_2}
\end{center}
\end{figure}

Since $t_0$ is conjugate to $t_1$ in $A[\tilde A_n]$, $\varphi'(t_0)$ is conjugate to $\varphi'(t_1)$ in $A[A_{n+1}]$. Hence $\varphi'(t_0)$ is of the form $\varphi'(t_0) = H_{a'}^\varepsilon T_{d'}^p T_{\partial\D}^q$, where $d'$ is a circle which bounds a disk containing $n+1$ marked points, and $a'$ is an arc inside this disk.
Moreover, if $\hat a'$ is the boundary of a regular neighborhood of $a'$, then $\SS(\varphi'(t_0)) = \{[\hat a'], [d']\}$.
Since $t_0t_2=t_2t_0$, we have $\varphi'(t_0)\,\varphi'(t_2)\,\varphi'(t_0)^{-1}=\varphi'(t_2)$. Hence, by Theorem~\ref{thm3_6}\,(3), $\varphi'(t_0) (\SS(\varphi'(t_2))) = \SS(\varphi'(t_2))$, and therefore $\varphi'(t_0) ([d]) = [d]$.
So, $[d]$ is a reduction class for $\varphi'(t_0)$.
Since $[d']$ is an essential reduction class for $\varphi'(t_0)$, it follows that $i([d'], [d]) = 0$.
So, we can assume without loss of generality that $d\cap d' = \varnothing$.
Each of the two circles $d$ and $d'$ bounds a disk containing $n+1$ marked points, and $d\cap d' = \varnothing$, hence $d'$ should be isotopic to $d$.
So, we can assume without loss of generality that $d = d'$.
Then $\varphi'(t_0) = H_{a'}^\varepsilon T_d^p T_{\partial\D}^q$.

Let $i \in \{2, \dots, n-1\}$.
We have $t_i\, t_0 = t_0\, t_i$, hence $\varphi'(t_i)\, \varphi'(t_0) = \varphi'(t_0)\, \varphi'(t_i)$.
Since, moreover, both $T_d$ and $T_{\partial\D}$ commute with $H_{a_i}$ and with $H_{a'}$, we have $H_{a_i} H_{a'} = H_{a'} H_{a_i}$.
By Proposition~\ref{prop3_5} we deduce that $a_i \cap a' = \varnothing$.
Let $i \in \{1, n\}$.
We have $t_i\, t_0\,t_i = t_0\, t_i\, t_0$, hence $\varphi'(t_i)\, \varphi'(t_0)\, \varphi'(t_i) = \varphi'(t_0)\, \varphi'(t_i)\, \varphi'(t_0)$.
Since, moreover, both $T_d$ and $T_{\partial\D}$ commute with $H_{a_i}$ and with $H_{a'}$, we have $H_{a_i} H_{a'} H_{a_i} = H_{a'} H_{a_i} H_{a'}$.
By Proposition~\ref{prop3_5} we deduce that $|a_i \cap a'|=1$ and $a_i \cap a' \subset \PP_{n+2}$.
It is easily seen that up to isotopy there exist exactly two arcs $a'$ satisfying $a'\cap d = \varnothing$, $a_i \cap a' = \varnothing$ for all $2 \le i \le n-1$, $|a_1 \cap a'|= |a_n \cap a'| =1$, $a_1 \cap a' \subset \PP_{n+2}$ and $a_n \cap a' \subset \PP_{n+2}$.
These two arcs are $c_0$ and $c_1$.
So, either $\varphi' (t_0) = H_{c_0}^\varepsilon T_d^p T_{\partial\D}^q$ or $\varphi' (t_0) = H_{c_1}^\varepsilon T_d^p T_{\partial\D}^q$. That is, either $\varphi(t_0) = g v_0^{\varepsilon} \Delta_Y^{2p} \Delta^{2q} g^{-1}$ or $\varphi(t_0) = g v_1^{\varepsilon} \Delta_Y^{2p} \Delta^{2q} g^{-1}$.

{\it Case 2: $p = 0$.}
Let $\varphi' = \conj_{g^{-1}} \circ \varphi$.
We have $\varphi'(t_i) = s_i'^\varepsilon \Delta^{2q} = H_{a_i}^\varepsilon T_{\partial\D}^q$ for all $1 \le i \le n$.
Since $t_0$ is conjugate to $t_1$ in $A[\tilde A_n]$, $\varphi'(t_0)$ is conjugate to $\varphi'(t_1)$ in $A[A_{n+1}]$. Hence $\varphi'(t_0)$ is of the form $\varphi'(t_0) = H_{a'}^\varepsilon T_{\partial\D}^q$, where $a'$ is an arc of $(\D, \PP_{n+2})$.

Let $i \in \{2, \dots, n-1\}$.
We have $t_i\, t_0 = t_0\, t_i$, hence $\varphi'(t_i)\, \varphi'(t_0) = \varphi'(t_0)\, \varphi'(t_i)$.
Since, moreover, $T_{\partial\D}$ commutes with $H_{a_i}$ and with $H_{a'}$, we have $H_{a_i} H_{a'} = H_{a'}  H_{a_i}$.
By Proposition~\ref{prop3_5} we deduce that $a_i \cap a' = \varnothing$.
Let $i \in \{1, n\}$.
We have $t_i\, t_0\, t_i = t_0\, t_i\, t_0$, hence $\varphi'(t_i)\, \varphi'(t_0)\, \varphi'(t_i) = \varphi'(t_0)\, \varphi'(t_i)\, \varphi'(t_0)$.
Since, moreover, $T_{\partial\D}$ commutes with $H_{a_i}$ and with $H_{a'}$, we have $H_{a_i} H_{a'} H_{a_i} = H_ {a'} H_{a_i} H_{a'}$.
By Proposition~\ref{prop3_5} we deduce that $|a_i \cap a'| = 1$ and $a_i \cap a' \subset \PP_{n+2}$. Since $a'$ does not intersect $a_i$ for any $2 \le i \le n-1$, we conclude that $a'\cap a_1$ is the singleton $\{p_0\}$, and $a' \cap a_n$ is the singleton $\{p_n\}$, so that $a'$ is an arc connecting $p_0$ to $p_n$.

The union $\left( \bigcup_{i=1}^n a_i \right) \cup a'$ bounds a disk $D$ embedded in $\D$.
This disk either contains $p_{n+1}$, or does not contain $p_{n+1}$.
If $D$ does not contain $p_{n+1}$, then it is easily seen that $a'$ is isotopic to $c_0$ or to $c_1$,  hence either $\varphi' (t_0) = H_{c_0}^\varepsilon T_{\partial\D}^q$ or $\varphi' (t_0) = H_{c_1}^\varepsilon T_{\partial\D}^q$, that is, either $\varphi(t_0) = g v_0^{\varepsilon} \Delta^ {2q} g^{-1}$ or $\varphi(t_0) = g v_1^{\varepsilon} \Delta^{2q} g^{-1}$.

Suppose now that $D$ contains $p_{n+1}$. For the rest of the proof we ``unmark'' point $p_{n+1}$ and consider isotopies of $\Omega=\big(\D,\{p_0,\dots,p_n\}\big)$, which are allowed to move point $p_{n+1}$. Under such isotopies, $a'$ is isotopic to either $c_0$ or $c_1$, as above. Let $\{F_t\colon\Omega\to\Omega\}_{t\in[0,1]}$ be such an isotopy, for which $F_0=\id$, $F_1(a')=c_k$, $(k\in\{0,1\})$, and $F_t$ is the identity on $\big(\bigcup_{i=1}^n a_i\big)\cup\partial\D$ for all $t\in[0,1]$. The reader should keep in mind that the disk $D$ can be embedded in $\D$ in a highly complicated manner, and the isotopy $F_t$ ``unwinds'' it to one of the two standard positions: $F_1(D)$ is bounded by $\bigcup_{i=1}^n a_i$ and $c_k$ (for $k=0$ or $1$), with $F_1(p_{n+1})$ belonging to $F_1(D)$. Now we choose another isotopy $\{F'_t\colon\Omega\to\Omega\}_{t\in[0,1]}$ such that $F'_0=\id$, $F'_1(c_k)=b_k$, $F'_t$ is the identity on $\big(\bigcup_{i=1}^n a_i\big)\cup\partial\D$ for all $t\in[0,1]$ and $F'_1(F_1(p_{n+1}))=p_{n+1}$. Obviously such isotopies exist. Let $\tilde F=F'_1\circ F_1$. Then $\tilde F\in \Homeo^+(\D,\PP_{n+2})$ and we let $h\in\MM(\D,\PP_{n+2})=A[A_{n+1}]$ be the mapping class $[\tilde F]$ represented by $\tilde F$. We see that $h ([a_i]) = [a_i]$ for all $1 \le i \le n$ and $h ([a']) = [b_k]$, for $k\in\{0,1\}$.
(Note that, by~\cite[Theorem 1.1]{Paris1}, we must have $h \in \langle \Delta_Y^2, \Delta^2 \rangle = \langle T_d, T_{\partial\D} \rangle$, but this fact is not needed for the proof.)
We have $h H_{a_i} h^{-1} = H_{a_i}$ for all $1 \le i \le n$ and $h H_{a'} h^{-1} = H_{b_k}$, hence
$(\conj_h \circ \varphi') (t_i) = s_i'^\varepsilon \Delta^{2q}$ for all $1 \le i\le n$ and $(\conj_h \circ \varphi') (t_0) = u_k^\varepsilon \Delta^{2q}$.
So, $\varphi(t_i) = g h^{-1} s_i'^{\varepsilon} \Delta^{2q} h g^{-1}$ for all $1 \le i \le n$ and $\varphi(t_0) = g h^{-1} u_k^{\varepsilon} \Delta^{2q} h g^{-1}$.
\end{proof}

\begin{proof}[Proof of Theorem~\ref{thm2_1}]
Let $\varphi\colon A[\tilde A_n] \to A[\tilde A_n]$ be a homomorphism.
Recall that $A[\tilde A_n]$ is viewed as a subgroup of $A[A_{n+1}]$, where $t_i = s_i'$ for all $1 \le i \le n$ and $t_0 = s_1' \ldots s_n' s_{n+1}'^2 s_n' s_{n+1}'^{-2} s_n'^{-1} \ldots s_1'^{-1}$.
Recall also that $Y = \{s_1', \dots, s_n'\} = \{t_1, \dots, t_n\}$, $\Delta_Y = \Delta_Y[A_{n+1}] = \Delta_Y[\tilde A_n]$ and $\Delta = \Delta[A_{n+1}]$.
By Proposition~\ref{prop4_1} we have one of the following three possibilities.
\begin{itemize}
\item[(1)]
$\varphi$ is cyclic.
\item[(2)]
There exist $g \in A[A_{n+1}]$, $k \in \{0, 1\}$, $\varepsilon \in \{\pm 1\}$ and $q \in \Z$ such that $\varphi(t_i) = g s_i'^{\varepsilon} \Delta^{2q} g^{-1}$ for all $1 \le i \le n$ and $\varphi(t_0) = g u_k^{\varepsilon} \Delta^{2q} g^{-1}$.
\item[(3)]
There exist $g \in A[A_{n+1}]$, $k \in \{0,1\}$, $\varepsilon \in \{\pm 1\}$ and $p, q \in \Z$ such that $\varphi(t_i) = g s_i'^{\varepsilon} \Delta_Y^{2p} \Delta^{2q} g^{-1}$ for all $1 \le i \le n$ and $\varphi(t_0) = g v_k^{\varepsilon} \Delta_Y^{2p} \Delta^{2q} g^{-1}$.
\end{itemize}

If $\varphi$ is cyclic, then there is nothing to prove.
So, we can assume that we are in Case~(2) or in Case~(3) of Proposition~\ref{prop4_1}.
More precisely, we suppose that there exist $g \in A[A_{n+1}]$, $w \in \{u_0, u_1, v_0, v_1\}$, $\varepsilon \in \{\pm 1\}$ and $p, q \in \Z$ such that $\varphi(t_i) = g s_i'^{\varepsilon} \Delta_Y^{2p} \Delta^{2q} g^{-1}$ for all $1 \le i \le n$, $\varphi(t_0) = g w^{\varepsilon} \Delta_Y^{2p} \Delta^{2q} g^{-1}$, and $p = 0$ if $w \in \{u_0, u_1\}$.

Recall the inclusions $A[\tilde A_n] \subset A[B_{n+1}] \subset A[A_{n+1}]$, where $r_i = s_i'$ for all $1 \le i \le n$ and $r_{n+1} = s_{n+1}'^2$. 
Recall also that $A[B_{n+1}] = \omega^{-1} (W_Y[A_{n+1}])$, where $\omega\colon A[A_{n+1}] \to W[A_{n+1}]$ is the   standard epimorphism which sends $s_i'$ to $s_i'$ for all $1 \le i \le n+1$ (see Remark~\ref{rem4} after Proposition~\ref{prop3_2}).
If we identify $A[A_{n+1}]$ with $\MM(\D,\PP_{n+2})$, $W[A_{n+1}]$ can be identified with the symmetric group $\Sym(\PP_{n+2})$ permuting the punctures, and hence $A[B_{n+1}]$ is isomorphic to the stabilizer of one puncture in $\MM(\D,\PP_{n+2})$, i.e.\ $A[B_{n+1}] = \{ f \in \MM(\D,\PP_{n+2}) \mid f(p_{n+1}) = p_{n+1}\}$. (In the braid group interpretation of $A[A_{n+1}]$, the subgroup $A[B_{n+1}]$ is exactly the subgroup that permutes all strands but one.)

We first prove that $g \in A[B_{n+1}]$.
By the above, it suffices to show that $g(p_{n+1}) = p_{n+1}$, or, equivalently, that $g^{-1}(p_{n+1}) = p_{n+1}$.
Suppose not, that is, suppose $g^{-1}(p_{n+1}) \neq p_{n+1}$.
Then there exists $i \in \{0, 1, \dots, n\}$ such that $g^{-1}(p_{n+1}) = p_i$.
Suppose $1 \le i \le n$.
On the one hand, $\varphi(t_i)(p_{n+1}) = p_{n+1}$, because $\varphi(t_i) \in A[\tilde A_n] \subset A[B_{n+1}]$.
On the other hand, $\varphi(t_i)$ is of the form $\varphi(t_i) = g s_i'^\varepsilon \Delta_Y^{2p} \Delta^{2q} g^{-1}$.
Note that, in the notation given in Figures~\ref{fig3_1} and \ref{fig4_2}, $s_i'^\varepsilon (p_i)=H_{a_i}^\varepsilon (p_i) = p_{i-1}$, $\Delta_Y^2 (p_i)=T_d (p_i)=p_i$ and $\Delta^2 (p_i) = T_{\partial\D} (p_i) = p_i$, hence
\[
\varphi(t_i) (p_{n+1}) = (g s_i'^\varepsilon \Delta_Y^{2p} \Delta^{2q} g^{-1}) (p_{n+1}) = (g s_i'^\varepsilon \Delta_Y^{2p} \Delta^{2q}) (p_i) = g(p_{i-1}) \neq p_{n+1}\,,
\]
because $g^{-1}(p_{n+1}) = p_i \neq p_{i-1}$.
This is a contradiction.
Now suppose $i = 0$.
On the one hand, $\varphi(t_1) (p_{n+1}) = p_{n+1}$, because $\varphi(t_1) \in A[\tilde A_n] \subset A[B_{n+1}]$.
On the other hand, $\varphi(t_1)$ is of the form $\varphi(t_1) = g s_1'^\varepsilon \Delta_Y^{2p} \Delta^{2q} g^{-1}$.
Note that $s_1'^\varepsilon (p_0)=H_{a_1}^\varepsilon (p_0) = p_{1}$, $\Delta_Y^2 (p_0)=T_d (p_0)=p_0$ and $\Delta^2 (p_0) = T_{\partial\D} (p_0) = p_0$, hence
\[
\varphi(t_1)(p_{n+1}) = (g s_1'^\varepsilon \Delta_Y^{2p} \Delta^{2q} g^{-1}) (p_{n+1}) = (g s_1'^\varepsilon \Delta_Y^{2p} \Delta^{2q}) (p_0) = g(p_1) \neq p_{n+1}\,,
\]
because $g^{-1}(p_{n+1}) = p_0 \neq p_1$.
This is a contradiction.
So, $g^{-1}(p_{n+1}) = p_{n+1}$, hence $g \in A[B_{n+1}]$.

Now, we prove that $q = 0$. We need the following lemma.
\begin{lem}\label{lem_delta} 
$\Delta^2=\Delta[A_{n+1}]^2=\Delta[B_{n+1}]$. 
\end{lem}
\begin{proof}[Proof of Lemma~\ref{lem_delta}]
It is known that 
\begin{align*}
\Delta [A_{n+1}]^2 &= (s_1' \ldots s_n' s_{n+1}'^2 s_n' \ldots s_1') (s_2' \ldots s_n' s_{n+1}'^2 s_n' \ldots s_2') \ldots (s_n' s_{n+1}'^2 s_n') s_{n+1}'^2\,,\\
\Delta [B_{n+1}] &= (r_1 \ldots r_n r_{n+1} r_n \ldots r_1) (r_2 \ldots r_n r_{n+1} r_n \ldots r_2) \ldots (r_n r_{n+1} r_n) r_{n+1}\,,
\end{align*}
(see \cite[Lemma 5.1]{CasPar1} for the first equality and \cite[Lemma 4.1]{Paris1} for the second one), hence $\Delta^2 = \Delta [A_{n+1}]^2 = \Delta [B_{n+1}]$. This completes the proof of the lemma.
\end{proof}

Let $z\colon A[B_{n+1}] \to \Z$ be the homomorphism which sends $r_i$ to $0$ for all $1 \le i \le n$ and sends $r_{n+1}$ to $1$. From the formula given in Section~\ref{sec2} we know that $\Delta[B_{n+1}]=(r_1\ldots r_nr_{n+1})^{n+1}$, and hence,  by Lemma~\ref{lem_delta}, $z(\Delta^2)=z(\Delta[B_{n+1}])=(n+1)z(r_{n+1})=n+1$. Recall that $\Ker(z) = A[\tilde A_n]$.
We have $z(\Delta_Y^2) = 0$, because $\Delta_Y^2 \in \langle r_1, \dots, r_n \rangle\subset A[\tilde A_n]$, and we have $z(\varphi(t_1)) = 0$, since $\varphi (t_1) \in A[\tilde A_n]$ as well.
The element $\varphi(t_1)$ is of the form $\varphi(t_1) = g r_1 \Delta_Y^{2p} \Delta^{2q} g^{-1}$, hence we also have
\[
0=z(\varphi(t_1)) = z(g) +q(n+1) -z(g) =q(n+1)\,,
\]
which is possible only if $q=0$.

Let $\rho_B = r_1 \ldots r_n r_{n+1}$ be the element of $A[B_{n+1}]$ defined at the beginning of Section~\ref{sec3}.
By Proposition~\ref{prop3_1}, there exist $g_1 \in A[\tilde A_n]$ and $m \in \Z$ such that $g = g_1 \rho_B^m$.
Moreover, $\rho_B^m f \rho_B^{-m} = \zeta^m(f)$ for all $f \in A[\tilde A_n]$.
We set $\varphi' = \zeta^{-m} \circ \conj_{g_1^{-1}} \circ \varphi$.
Then $\varphi'(t_i) = s_i'^\varepsilon \Delta_Y^{2p} = t_i^\varepsilon \Delta_Y^{2p}$ for all $1 \le i \le n$, and $\varphi'(t_0) = w^\varepsilon \Delta_Y^{2p}$.

From Proposition~\ref{prop4_1} we see that if $w=u_0$ or $u_1$, then there is no $\Delta_Y^{2p}$ factor in the formulas for $\varphi$, i.e.\ $p=0$. Now, if $w = u_0$ and $\varepsilon = 1$, then $\varphi' = \id$, hence $\varphi = \conj_{g_1} \circ \zeta^m \circ \varphi' = \conj_{g_1} \circ \zeta^m$.
If $w = u_0$ and $\varepsilon = -1$, then $\varphi' = \mu$, hence $\varphi = \conj_{g_1} \circ \zeta^m \circ \varphi' = \conj_{g_1} \circ (\zeta^m \circ \mu)$.
If $w = u_1$ and $\varepsilon = 1$, then $\varphi' =\conj_{\Delta_Y^{-1}} \circ \zeta \circ \eta$, hence $\varphi = \conj_{g_1} \circ \zeta^m \circ \varphi' = \conj_h \circ (\zeta^{m+1} \circ \eta)$, where $h = g_1\cdot \zeta^m(\Delta_Y^{-1})$.
If $w = u_1$ and $\varepsilon = -1$, then $\varphi' = \conj_{\Delta_Y^{-1}} \circ \zeta \circ \eta \circ \mu$, hence $\varphi = \conj_{g_1} \circ \zeta^m \circ \varphi' = \conj_h \circ (\zeta^{m+1} \circ \eta \circ \mu)$, where $h = g_1\cdot \zeta^m(\Delta_Y^{-1})$. We conclude that if $w\in\{u_0,u_1\}$, then $\varphi$ has the form stated in Case~(2) of the theorem.

To understand the case $w\in\{v_0, v_1\}$ we must first observe that $\mu(v_0) = v_1^{-1}$ and $\mu(v_1) = v_0^{- 1}$.
If $w = v_0$ and $\varepsilon = 1$, then $\varphi' = \alpha_p$, hence $\varphi = \conj_{g_1} \circ \zeta^m \circ \varphi' = \conj_{g_1} \circ \zeta^m \circ \alpha_p$.
If $w = v_0$ and $\varepsilon = -1$, then $\varphi' = \mu \circ \beta_{-p}$, hence $\varphi = \conj_{g_1} \circ \zeta^m \circ \varphi' = \conj_{g_1} \circ (\zeta^m \circ \mu) \circ \beta_{-p}$.
If $w = v_1$ and $\varepsilon = 1$, then $\varphi' = \beta_p$, hence $\varphi = \conj_{g_1} \circ \zeta^m \circ \varphi' = \conj_{g_1} \circ \zeta^m \circ \beta_p$.
If $w = v_1$ and $\varepsilon = -1$, then $\varphi' = \mu \circ \alpha_{-p}$, hence $\varphi = \conj_{g_1} \circ \zeta^m \circ \varphi' = \conj_{g_1} \circ (\zeta^m \circ \mu) \circ \alpha_{-p}$. We conclude that if $w\in\{v_0,v_1\}$, then $\varphi$ has the form stated in Cases~(3) and (4) of the theorem. 

This completes the proof of Theorem~~\ref{thm2_1}.
\end{proof}

\begin{proof}[Proof of Proposition~\ref{prop2_2}]
Let $\psi, \psi' \in \Aut(A[\tilde A_n])$ and $p, q \in \Z$ such that $\psi \circ \alpha_p = \psi' \circ \alpha_q$.
Up to replacing $\psi$ with $\psi'^{-1} \circ \psi$ if necessary, we can assume that $\psi' = \id$.
Let $x\colon A[\tilde A_n] \to \Z$ be the homomorphism which sends $t_i$ to $1$ for all $0 \le i \le n$.
By Theorem~\ref{thm2_1} there exist $\psi_1 \in \Aut^*(\tilde A_n)$ and $g \in A[\tilde A_n]$ such that $\psi = \conj_g \circ \psi_1$.
From this decomposition it follows that either $x(\psi(t_i)) = 1$ for all $0 \le i \le n$, or $x(\psi(t_i)) = -1$ for all $0 \le i \le n$.
This implies that either $(x \circ \psi \circ \alpha_p) (t_1) = 1+pn(n+1)$, or $(x \circ \psi \circ \alpha_p) (t_1) = -1-pn (n+1)$.
Since $n(n+1) \ge 20$ and $(x \circ \alpha_q) (t_1) = 1+qn(n+1)$, it follows that $p = q$ and $(x \circ \psi \circ \alpha_p) (t_1) = 1+pn(n+1)$.

Let $\psi, \psi' \in \Aut(A[\tilde A_n])$ and $p, q \in \Z$.
We prove in the same way that, if $\psi \circ \beta_p = \psi' \circ \beta_q$, then $p = q$.
Similarly, if $\psi \circ \alpha_p = \psi' \circ \beta_q$, then $p = q$.
It remains to show that $\psi \circ \alpha_p \neq \psi' \circ \beta_p$.

Suppose $\psi \circ \alpha_p = \psi' \circ \beta_p$.
Up to replacing $\psi$ with $\psi'^{-1} \circ \psi$ if necessary, we can assume that $\psi' = \id$.
For each $1 \le i \le n$ we have
\[
\psi (t_i \Delta_Y^{2p}) = (\psi \circ \alpha_p) (t_i) = \beta_p (t_i) = t_i \Delta_Y^{2p}\,,
\]
hence, since $\Delta_Y^2$ commutes with $t_i$ for all $1 \le i \le n$,
\begin{gather*}
(\psi \circ \alpha_p) (t_0) = \psi (v_0 \Delta_Y^{2p}) = \psi \big( (t_1 \ldots t_{n-1} t_n t_{n-1}^{-1} \ldots t_1^{-1}) \Delta_Y^{2p} \big) = \\
\psi \big( (t_1 \Delta_Y^{2p}) \ldots (t_{n-1} \Delta_Y^{2p}) (t_n \Delta_Y^{2p}) (t_{n-1} \Delta_Y^{2p})^{-1}
\ldots (t_1 \Delta_Y^{2p})^{-1} \big) = \\
(t_1 \Delta_Y^{2p}) \ldots (t_{n-1}\Delta_Y^{2p}) (t_n\Delta_Y^{2p}) (t_{n-1}\Delta_Y^{2p})^{-1} \ldots (t_1 \Delta_Y^{2p})^{-1} = \\
(t_1 \ldots t_{n-1} t_n t_{n-1}^{-1} \ldots t_1^{-1}) \Delta_Y^{2p} = v_0 \Delta_Y^{2p}\,.
\end{gather*}
This is a contradiction because $\beta_p (t_0) = v_1 \Delta_Y^{2p}$, and $v_0 \neq v_1$ by Corollary~\ref{cor3_6}, since the respective half-twists $H_{c_0}=v_0$ and $H_{c_1}=v_1$ correspond to non-isotopic arcs $c_0$ and $c_1$ in $(\D,\PP_{n+2})$, see Figure~\ref{fig4_1}.
\end{proof}

%%%%%%%%%%

%%%%%%%%%%

\end{document}